\documentclass[10pt, reqno]{amsart}
\usepackage{amsaddr}
\usepackage{latexsym,amsfonts}
\usepackage{amssymb,amsthm,upref,amscd}
\usepackage[T1]{fontenc}
\usepackage{times}
\usepackage{amscd}
\usepackage{enumerate}
\usepackage{mathrsfs}
\usepackage{amsmath}
\usepackage{color}
\usepackage{soul}
\usepackage{indentfirst}
\usepackage{comment}
\PassOptionsToPackage{reqno}{amsmath}

\newtheorem{thm}{Theorem}[section]

\newtheorem{lem}{Lemma}[section]

\newtheorem{rem}{Remark}[section]
\theoremstyle{notation}
\newtheorem*{notation}{Notation}

\newcommand{\R}{\mathbb{R}}
\newcommand{\C}{\mathbb{C}}
\numberwithin{equation}{section}

\newcommand{\eps}{\epsilon}

\newcommand{\wto}{\rightharpoonup}

\makeatletter \@addtoreset{equation}{section} \makeatother

\usepackage[textwidth=138mm,
textheight=215mm,
left=30mm,
right=30mm,
top=25.4mm,
bottom=25.4mm,
headheight=2.17cm,
headsep=4mm,
footskip=12mm,
heightrounded,]{geometry}

\usepackage[colorlinks=true,urlcolor=blue,
citecolor=black,linkcolor=black,linktocpage,pdfpagelabels,
bookmarksnumbered,bookmarksopen]{hyperref}
\newcounter{const}
\author[T. Gou]{Tianxiang Gou}
\address[T. Gou]{%
\centerline{School of Mathematics and Statistics, Xi’an Jiaotong University,}
\centerline{710049, Xi’an, Shaanxi, China}}



\email{tianxiang.gou@xjtu.edu.cn}
\title[Uniqueness of ground states]{Uniqueness of ground states to fractional nonlinear elliptic equations with harmonic potential}
\thanks{{\it Acknowledgments}. The author was supported by the National Natural Science Foundation of China (No. 12101483) and the Postdoctoral Science Foundation of China (No. 2021M702620). The author would like to thank warmly the referee for the helpful and constructive comments to improve the manuscript.}
\thanks{{\it Conflict of interest statement}. The author declares that there is no conflict of interests.}

\begin{document}

\begin{abstract}
In this paper, we prove the uniqueness of ground states to the following fractional nonlinear elliptic equation with harmonic potential,
$$
(-\Delta)^s u+ \left(\omega+|x|^2\right) u=|u|^{p-2}u \quad \mbox{in}\,\, \R^n,
$$
where $n \geq 1$, $0<s<1$, $\omega>-\lambda_{1,s}$, $2<p<\frac{2n}{(n-2s)^+}$, $\lambda_{1,s}>0$ is the lowest eigenvalue of $(-\Delta)^s + |x|^2$. The fractional Laplacian $(-\Delta)^s$ is characterized as $\mathcal{F}((-\Delta)^{s}u)(\xi)=|\xi|^{2s} \mathcal{F}(u)(\xi)$ for $\xi \in \R^n$, where $\mathcal{F}$ denotes the Fourier transform. This solves an open question in \cite{SS} concerning the uniqueness of ground states.
\end{abstract}

\maketitle

\thispagestyle{empty}

\section{Introduction}

In this paper, we study the uniqueness of ground states to the following fractional nonlinear elliptic equation with harmonic potential,
\begin{align} \label{equ}
(-\Delta)^s u+ \left(\omega+|x|^2\right) u=|u|^{p-2}u \quad \mbox{in}\,\, \R^n,
\end{align}
where $n \geq 1$, $0<s<1$, $\omega>-\lambda_{1,s}$, $2<p<2_s^*:=\frac{2n}{(n-2s)^+}$ and $\lambda_{1,s}>0$ is the lowest eigenvalue of $(-\Delta)^s + |x|^2$, which is defined by
\begin{align} \label{eigenvalue}
\lambda_{1, s}:=\inf_{u \in \Sigma_s} \left\{ \left \langle\left((-\Delta)^s + |x|^2\right) u, u \right \rangle :  \|u\|_2 =1 \right\}, \quad \Sigma_s:=H^s(\R^n) \cap L^2(\R^n; |x|^2 \, dx).
\end{align}
The fractional Laplacian $(-\Delta)^s$ is characterized as $\mathcal{F}((-\Delta)^{s}u)(\xi)=|\xi|^{2s} \mathcal{F}(u)(\xi)$ for $\xi \in \R^n$, where $\mathcal{F}$ denotes the Fourier transform defined by
$$
\mathcal{F}(u)(\xi):=\int_{\R^n} e^{-2 \pi \textnormal{i} x \cdot \xi} u(x)\,dx.
$$
For $0<s<1$, the fractional Sobolev space $H^s(\R^n)$ is defined by
$$
H^s(\R^n):=\left\{u\in L^2(\R^n) \, : \, \int_{\R^n}\left(1+|\xi|^{2s}\right) |\mathcal{F}(u)|^2 \, d\xi<\infty\right\}
$$
equipped with the norm
$$
\|u\|^2_{H^s}=\int_{\R^n}\left(1+|\xi|^{2s}\right) |\mathcal{F}(u)|^2 \, d\xi.
$$

The problem under consideration arises in the study of standing waves to the following time-dependent Schr\"odinger equation,
\begin{align} \label{equt}
\textnormal{i} \partial_t \psi+(-\Delta)^s \psi+ |x|^2\psi=|\psi|^{p-2}\psi \quad \mbox{in}\,\, \R \times \R^n.
\end{align}
Here a standing wave to \eqref{equt} is a solution of the form
$$
\psi(t,x)=e^{-\textnormal{i}\omega t} u(x), \quad \omega \in \R.
$$
It is simple to see that $\psi$ is a solution to \eqref{equt} if and only if $u$ is a solution to \eqref{equ}. The equation \eqref{equ} is of particular interest in fractional quantum mechanics and originates from the early work of Laskin \cite{L1,L2}.

For the case $s=1$, the uniqueness of ground states to \eqref{equ} was achieved in \cite{HO1, HO2}. However, for the case $0<s<1$, the uniqueness of ground states to \eqref{equ} is open so far. The aim of this paper is to make a contribution towards this direction.

In the present paper, we are only concerned with the uniqueness of ground states to \eqref{equ}, the existence of which is a simple consequence of the use of mountain pass theorem, see \cite[Theorem 1.15]{W}, and the fact that $\Sigma_s$ is compactly embedded into $L^q(\R^n)$ for any $2 \leq q<2_s^*$, see \cite[Lemma 3.1]{DH}. Moreover, in view of the maximum principle, we can further obtain that any ground state to \eqref{equ} is positive. The main result of the paper reads as follows.

\begin{thm} \label{thm1}
Let $n \geq 1$, $0<s<1$, $\omega>-\lambda_{1,s}$ and $2<p<2_s^*$. Then ground state to \eqref{equ} is unique.
\end{thm}

Due to the nonlocal feature of the fractional Laplacian operator, the well-known ODE techniques often adapted to discuss the uniqueness of ground states to nonlinear elliptic equations with $s=1$ are not applicable to our problem. Therefore, to establish Theorem \ref{thm1}, we shall make use of the scheme developed in \cite{FL, FLS}.

\begin{rem}
Theorem \ref{thm1} answers an open question posed in \cite{SS} with respect to the uniqueness of ground states to \eqref{equ}, which also extends the uniqueness results in \cite{HO1, HO2} for $s=1$ to the case $0<s<1$. 
\end{rem}

\begin{notation}
For $1 \leq q \leq \infty$, we denote by $\|\cdot\|_q$ the standard norm in the Lebesgue space $L^q(\R^n)$. Moreover, we use $X \lesssim Y$ to denote that $X \leq C Y$ for some proper constant $C>0$ and we use $X \sim Y$ to denote $X \lesssim Y$ and $Y \lesssim X$.
\end{notation}

\section{Proof of Theorem \ref{thm1}}

In this section, we are going to establish Theorem \ref{thm1}. To do this, we first present the nondegeneracy of ground states.

\begin{lem} \label{nondegeneracy}
Let $n \geq 1$, $0<s<1$, $\omega>-\lambda_{1,s}$ and $2<p<2_s^*$. Let $u \in \Sigma_s$ be a ground state to \eqref{equ}. Then the linearized operator
$$
\mathcal{L}_{+, s}:=(-\Delta)^s + \left(\omega+|x|^2\right) -(p-1)|u|^{p-2}
$$
has a trivial kernel.
\end{lem}
\begin{proof}
To prove this lemma, one can follow closely the line of the proof of \cite[Theorem 2]{SS}. Let us now sketch the proof. First we observe that $\mathcal{L}_{+,s} \mid_{\{u\}^{\bot}} \geq 0$. On the other hand, we find that
$$
\langle \mathcal{L}_{+,s} u, u \rangle =-(p-2) \int_{\R^n} |u|^p <0.
$$
It then follows that $\mathcal{L}_{+,s}$ has only one negative eigenvalue. From \cite[Proposition 7]{SS}, we actually know that the eigenvalue is simple. Using spherical harmonics and the representations of fractional Schr\"odinger operators introduced in \cite{SS}, we can write that
$$
\mathcal{L}_{+,s}=\bigoplus_{l=0}^{\infty} \mathcal{L}_{+,s, l}:=\mathcal{L}_{+,s, 0} \bigoplus \mathcal{L}_{+,s,  \geq 1},
$$
where the operator $\mathcal{L}_{+,s, l}$ acting on $L_{rad}^2(\R^n)$ is given by
$$
\mathcal{L}_{+,s,l}:=\left(-\partial_{rr} -\frac{n-1}{r} \partial_r +\frac{l(l+n-2)}{r^2}\right)^{s} + \left(\omega+|x|^2\right) -(p-1)|u|^{p-2}, \quad  l=0, 1, \cdots, k, \cdots.
$$
It is clear that
$$
\sigma(\mathcal{L}_{+,s})=\bigcup_{l=0}^{\infty} \sigma(\mathcal{L}_{+,s, l}),
$$
$$
\mathcal{L}_{+,s, 0} <\mathcal{L}_{+,s, 1}< \cdots <\mathcal{L}_{+,s, k} < \cdots.
$$
At this point, to conclude the proof, we only need to verify that the second smallest eigenvalue of $\mathcal{L}_{+,s, 0}$ is positive and $\mathcal{L}_{+,s, \geq 1} \geq \delta >0$. This can be achieved by applying \cite[Propositions 8-9]{SS}. Thus the proof is completed.
\end{proof}

In order to establish Theorem \ref{thm1}, we shall closely follow the strategies developed in \cite{FL, FLS}. For this, we now introduce some notations. Let $n \geq 1$, $0<s<1$, $\omega>-\lambda_{1,s}$ and $2<p<2_s^*$. Define
$$
X_p:=\left\{ u \in L^2(\R^n) \cap L^p(\R^n) : u \,\, \mbox{is radially symmetric and real-valued and}\,\, \|x u\|_2<+\infty \right\}
$$
equipped with the norm
$$
\|u\|_{X_p}:=\|u\|_2+\|u\|_p +\|x u\|_2.
$$

\begin{lem} \label{regularity}
Let $n \geq 1$, $0<s<1$, $\omega>-\lambda_{1,s}$ and $2<p<2_s^*$ and $u \in X_p$ be a solution to \eqref{equ}. Then $u \in H^{s}(\R^n)$.
\end{lem}
\begin{proof}
First we show that $u \in H^1(\R^n)$. Since $v \in X_p$ be a solution to \eqref{equ}, then
\begin{align} \label{equ1}
(-\Delta)^s u+ \left(\omega+|x|^2\right) u +2 \lambda u = 2\lambda u+|u|^{p-2} u,
\end{align}
where $\lambda>0$ satisifies $\omega+\lambda>0$.
Note that
$$
(-\Delta)^s + \left(\omega+|x|^2\right) +2 \lambda > (-\Delta)^s + \lambda>0.
$$
This leads to
\begin{align} \label{decr}
0<\left((-\Delta)^s + \left(\omega+|x|^2\right) +2 \lambda \right)^{-1} < \left((-\Delta)^s + \lambda \right)^{-1}.
\end{align}
It then follows from Young's inequality that
\begin{align} \label{ineq1}
\begin{split}
\left\|\left((-\Delta)^s + \left(\omega+|x|^2\right) +2 \lambda \right)^{-1}  u\right\|_2 
&\leq \left\|\left((-\Delta)^s + \lambda \right)^{-1} u\right\|_2=\left\|\mathcal{K} \ast u \right\|_2 \lesssim \|u\|_2 \lesssim \|u\|_{H^{-s}},
\end{split}
\end{align}
where $H^{-s}(\R^n)$ denotes the dual space of $H^s(\R^n)$ and $\mathcal{K}$ is the fundamental solution to the equation
\begin{align} \label{bequ}
(-\Delta)^s u + \lambda u=0
\end{align}
and $\mathcal{K} \in L^1(\R^n)$ by \cite[Lemma C. 1]{FLS}. 
This indicates that the operator $\left((-\Delta)^s + \left(\omega+|x|^2\right) +2 \lambda\right)^{-1}$
maps $H^{-s}(\R^n)$ to $L^2(\R^n)$. Using dual theory, we then see that $\left((-\Delta)^s + \left(\omega+|x|^2\right) +2 \lambda\right)^{-1}$ maps $L^2(\R^n)$ to $H^{s}(\R^n)$. 
Observe that
\begin{align} \label{ineq2}
\begin{split}
\left\|\left((-\Delta)^s + \left(\omega+|x|^2\right) +2 \lambda \right)^{-1}  u\right\|_{H^s}
&\leq \left\|\left((-\Delta)^s + \lambda \right)^{-1} u\right\|_{H^s} \lesssim \|u\|_{H^{-s}} \lesssim \|u\|_{p'},
\end{split}
\end{align}
where the last inequality is from the dual to the Sobolev embedding $\|u\|_p \lesssim \|u\|_{H^s}$.
This indicates that the operator 
$\left((-\Delta)^s + \left(\omega+|x|^2\right) +2 \lambda\right)^{-1}$
maps $L^{p'}(\R^n)$ to $H^{s}(\R^n)$. 
In fact, this can 
Observe that 
$$
u=\left((-\Delta)^s + \left(\omega+|x|^2\right) +2 \lambda \right)^{-1} \left(2\lambda u+|u|^{p-2} u\right) 
$$
and 
$$
2\lambda u+|u|^{p-2} u \in L^2(\R^n) + L^{p'}(\R^n).
$$ 
Then the desired result follows. 
This completes the proof.
\end{proof}

\begin{lem} \label{monotonicity}
Let $s_n \to s$ as $n \to \infty$, then $\lambda_{1,s_n} \to \lambda_{1,s}$ as $n \to \infty$.
\end{lem}
\begin{proof}
To prove this, we only need to show that $A_{s_n} \to A_s$ in the norm-resolvent sense as $n \to \infty$, where
$$
A_{s_n}:=(-\Delta )^{s_n} + |x|^2, \quad A_s:=(-\Delta )^s + |x|^2.
$$
Let $z \in \C$ be such that $\mbox{Im}\, z \neq 0 $, then
\begin{align*}
A_{s_n} +z=(-\Delta )^{s_n} + |x|^2+z&=(-\Delta )^s + |x|^2+z +(-\Delta )^{s_n}-(-\Delta )^s\\
&=\left(1+\left((-\Delta )^{s_n}-(-\Delta )^s\right)\left(A_s+z\right)^{-1}\right)\left(A_s+z\right).
\end{align*}
Then we see that
\begin{align} \label{conv1}
\left(A_{s_n} +z\right)^{-1} -\left(A_s+z\right)^{-1}=\left(A_s+z\right)^{-1} \left(\left(1+\left((-\Delta )^{s_n}-(-\Delta )^s\right)\left(A_s+z\right)^{-1}\right)^{-1}-1\right).
\end{align}
Note that
$$
\left\|\left((-\Delta )^{s_n}-(-\Delta )^s\right)\left(A_s+z\right)^{-1} \right\|_{L^2 \to L^2}=o_n(1).
$$
In addition, we see that $\left(A_s +z\right)^{-1}$ is bounded from $L^2(\R^n)$ to $L^2(\R^n)$. As a consequence, from \eqref{conv1}, we can conclude that
$$
\left\|\left(A_{s_n} +z\right)^{-1} -\left(A_s+z\right)^{-1} \right\|_{L^2 \to L^2}=o_n(1).
$$
This completes the proof.
\end{proof}

\begin{lem} \label{bifurcation}
Let $0<s_0<1$ and $2<p<2_{s_0}^*$. Suppose that $u_0 \in X_p$ solves \eqref{equ1} with $s=s_0$ such that the linearized operator
$$
\mathcal{L}_{+, s_0}=(-\Delta)^{s_0} + \left(\omega+|x|^2\right) -(p-1)|u_0|^{p-2}
$$
has a trivial kernel on $L^2_{rad}(\R^n)$, where $w>-\lambda_{1, s_0}$. Then there exist $\delta_0>0$ and a map $u \in C^1(I; X_p)$ with $I=[s_0, s_0+ \delta_0)$ such that
\begin{itemize}
\item[$(\textnormal{i})$] $u_s$ solves \eqref{equ} for $s \in I$, where $u_s:=u(s)$ for $s \in I$.
\item [$(\textnormal{ii})$] There exists $\eps>0$ such that $u_s$ is the unique solution to \eqref{equ} for $s \in I$ in the neighborhood 
$$
\left\{u \in X_p: \|u-u_0\|_{X_p}< \eps \right\},
$$
where $u_{s_0}=u_0$.
\end{itemize}
\end{lem}
\begin{proof}
Let $\delta_0>0$ be a small constant to be determined later and $\lambda_{1,s}>0$ be the lowest eigenvalue of $\left(-\Delta\right)^s + |x|^2$ for $s \in [s_0, s_0+\delta_0)$. Define a mapping $F: X_p \times [s_0, s_0+ \delta_0) \to X_p$ by
\begin{align*}
F(u, s):= 
\begin{aligned}
u-\left(\left(-\Delta\right)^s + \left(\omega+|x|^2\right) + 2 \lambda\right)^{-1} \left(2\lambda u +|u|^{p-2}u \right), 
\end{aligned}
\end{align*}
where $\omega>0$ satisfies $\omega>-\lambda_{1,s}$ and $\lambda>0$ satisfies $\lambda_{1,s}<\lambda$ for any $s \in [s_0, s_0+\delta_0)$. Due to $\omega>-\lambda_{1,s_0}$, by Lemma \ref{monotonicity}, then there exists $\delta_0>0$ small such that $\omega>-\lambda_{1,s}$ is valid for any $s \in [s_0, s_0+\delta_0)$. Moreover, observe that $\Sigma_1 \subset \Sigma_s$, then
$$
\lambda_{1,s} \leq \inf_{u \in \Sigma_s} \left\{ \left \langle \left(-\Delta+ |x|^2\right) u, u \right \rangle : \|u\|_2 =1 \right\} \leq \lambda_{1,1},
$$
where $\lambda_{1,1}>0$ is defined by
\begin{align*}
\lambda_{1, 1}:=\inf_{u \in \Sigma_1} \left\{ \left \langle\left(-\Delta + |x|^2\right) u, u \right \rangle :  \|u\|_2=1 \right\}.
\end{align*}
This then justifies that there exists $\lambda>0$ such that $\lambda_{1,s}<\lambda$ for any $s \in [s_0, s_0+\delta_0)$.

First we check that $F$ is well-defined. As an immediate consequence of the proof of Lemma \ref{regularity}, we see that $F(u, s) \in L^2(\R^n) \cap L^p(\R^n)$ for any $u \in X_p$ and $s\in [s_0, s_0 +\delta_0)$. Let us now check that $F(u, s) \in L^2(\R^n; |x|^2 \, dx)$ for any $u \in X_p$ and $s\in [s_0, s_0 +\delta)$. Define
$$
f:=\left(\left(-\Delta\right)^s + \left(\omega+|x|^2\right) + 2 \lambda\right)^{-1} \left(2\lambda u +|u|^{p-2}u \right).
$$
As the proof of Lemma \ref{regularity}, we find that $f \in H^s(\R^n)$. This further gives that
$$
\left(-\Delta\right)^s f+ \left(\omega+|x|^2\right) f+ 2 \lambda f=2\lambda u +|u|^{p-2}u.
$$
Therefore, we have that
\begin{align*}
\int_{\R^n} |\left(-\Delta\right)^{\frac s 2}f|^2 \, dx + \int_{\R^n}\left(\omega+|x|^2\right)|f|^2 \, dx + 2 \lambda \int_{\R^n}|f|^2 \, dx&=2\lambda \int_{\R^n} u f \,dx + \int_{\R^n} |u|^{p-2} u f \,dx \\
& \leq 2 \|u\|_2 \|f\|_2 + \|u\|_p^{p-1}\|f\|_p<+\infty,
\end{align*}
where we used H\"older's inequality for the inequality. It then leads to the desired result.

To apply the implicit function theorem, we are going to check that $F$ is of class $C^1$. First we show that $\frac{\partial F}{\partial u}$ exists and 
$$
\frac{\partial F}{\partial u}=1-\left(\left(-\Delta\right)^s + \left(\omega+|x|^2\right) +2 \lambda \right)^{-1} \left(2\lambda +(p-1)|u|^{p-2}\right).
$$
For simplicity, we shall define 
$$
G(u, s):=\left(\left(-\Delta\right)^s + \left(\omega+|x|^2\right) +2 \lambda\right)^{-1} \left(2\lambda u +|u|^{p-2} u\right).
$$
Indeed, it suffices to prove that $\frac{\partial G}{\partial u}$ exists and 
$$
\frac{\partial G}{\partial u}=\left(\left(-\Delta\right)^s + \left(\omega+|x|^2\right) + 2 \lambda \right)^{-1} \left(2\lambda +(p-1)|u|^{p-2}\right).
$$
Observe that, for any $h \in X_p$,
\begin{align*}
&\left\|G(u+h, s)-G(u, s)-\frac{\partial G}{\partial u}(u, s) h\right\|_{L^2 \cap L^p} \\
&=\left\|\left(\left(-\Delta\right)^s + \left(\omega+|x|^2\right) +2 \lambda\right)^{-1}\left(|u+h|^{p-2} (u+h)-|u|^{p-2} u-(p-1)|u|^{p-2} h\right) \right\|_{L^2 \cap L^p} \\
& \leq \left\|\left(\left(-\Delta\right)^s + \lambda \right)^{-1}\left(|u+h|^{p-2} (u+h)-|u|^{p-2} u-(p-1)|u|^{p-2} h\right) \right\|_{L^2 \cap L^p} \\
& \lesssim \left\||u+h|^{p-2} (u+h)-|u|^{p-2} u-(p-1)|u|^{p-2} h\right\|_{\frac{p}{p-1}}=o(\|h\|_p)=o(\|h\|_{L^2 \cap L^p}),
\end{align*}
where we used the fact that the fundamental solution $\mathcal{K}$ to \eqref{bequ} satisfies $\mathcal{K} \in L^{\frac p 2}(\R^n) \cap L^{\frac{2p}{p+2}} (\R^n)$ and Young's inequality.
Define
$$
g:=\left(\left(-\Delta\right)^s + \left(\omega+|x|^2\right) +2 \lambda\right)^{-1}\left(|u+h|^{p-2} (u+h)-|u|^{p-2} u-(p-1)|u|^{p-2} h\right).
$$
Since
$$
|u+h|^{p-2} (u+h)-|u|^{p-2} u-(p-1)|u|^{p-2} h \in L^{p'}(\R^n),
$$
then $g \in H^s(\R^n)$ by arguing as the proof of Lemma \ref{regularity}. Then we write
$$
\left(-\Delta\right)^s g + \left(\omega+|x|^2\right) g +2 \lambda g= |u+h|^{p-2} (u+h)-|u|^{p-2} u-(p-1)|u|^{p-2} h.
$$
It then follows that
\begin{align*}
&\int_{\R^n} |\left(-\Delta\right)^{\frac s 2}g|^2 \, dx + \int_{\R^n}\left(\omega+|x|^2\right)|g|^2 \, dx + 2 \lambda \int_{\R^n}|g|^2 \, dx \\
& = \int_{\R^n} \left(|u+h|^{p-2} (u+h)-|u|^{p-2} u-(p-1)|u|^{p-2} h\right) g \, dx=o(\|h\|_p)\|g\|_p.
\end{align*}
Using the fact that $H^s(\R^n)$ is continuously embedded into $L^p(\R^n)$ and Young's inequality, we then obtain that 
$$
\|g\|_{L^2(\R^n; |x|^2 \,dx)} \lesssim o(\|h\|_2).
$$ 
Consequently, there holds that
\begin{align*}
&\left\|G(u+h, s)-G(u, s)-\frac{\partial G}{\partial u}(u, s) h\right\|_{L^2(\R^n; |x|^2\,dx)} \\
&=\left\|\left(\left(-\Delta\right)^s + \left(\omega+|x|^2\right) +2 \lambda\right)^{-1}\left(|u+h|^{p-2} (u+h)-|u|^{p-2} u-(p-1)|u|^{p-2} h\right) \right\|_{L^2(\R^n; |x|^2\,dx)}\\
&=\|g\|_{L^2(\R^n; |x|^2 \,dx)} \lesssim o(\|h\|_2). 
\end{align*}
Thus we conclude that
$$
\left\|G(u+h, s)-G(u, s)-\frac{\partial G}{\partial u}(u, s) h\right\|_{X_p} \lesssim o(\|h\|_{X_p}).
$$
The desired result follows.

Next we are going to verify that $\frac{\partial F}{\partial u}$ is continuous. Indeed, it suffices to show that $\frac{\partial G}{\partial u}$ is continuous. For this aim, we shall demonstrate that, for any $\eps>0$, there exists $\delta>0$ such that $\|u-\tilde{u}\|_{X_p} +|s-\tilde{s}| <\delta$, then, for any $h \in X_p$,
\begin{align} \label{conti}
\left\|\left(\frac{\partial G}{\partial u}(u, s)-\frac{\partial G}{\partial u}(\tilde{u}, \tilde{s})\right) h\right\|_{X_p}<\eps \|h\|_{X_p}.
\end{align}
Observe that
\begin{align*}
\left(\frac{\partial G}{\partial u}(u, s)-\frac{\partial G}{\partial u}(\tilde{u}, \tilde{s})\right) h
&= \left(A_s -A_{\tilde{s}}\right) \left(2\lambda +(p-1)|u|^{p-2} \right) h+ A_{\tilde{s}} \left(2 \lambda +(p-1)\left(|u|^{p-2} -|\tilde{u}|^{p-2}\right)\right) h,
\end{align*}
where
$$
A_s:=\left(\left(-\Delta\right)^s + \left(\omega+|x|^2\right) + 2 \lambda \right)^{-1}, \quad A_{\tilde{s}}:=\big(\left(-\Delta\right)^{\tilde{s}} + \left(\omega+|x|^2\right) + 2 \lambda\big)^{-1}.
$$
Note that
$$
\|f\|_{L^2 \cap L^p} \lesssim \left\|\big((-\Delta)^{\frac {s_p} {2}} + 1 \big) f\right\|_2, \quad 0<s_p:=\frac{(p-2)n}{2p}<s.
$$
Then, by Plancherel’s identity, the mean value theorem and Young's inequality, there holds that
\begin{align*}
\left\|\left(A_s -A_{\tilde{s}}\right) \left(2\lambda +(p-1)|u|^{p-2} \right) h\right\|_{L^2 \cap L^p} &\lesssim \left\|\big((-\Delta)^{\frac {s_p} {2}} + 1 \big) \left(A_s -A_{\tilde{s}}\right) \left(2\lambda +(p-1)|u|^{p-2} \right) h\right\|_2 \\
& \lesssim  |s-\tilde{s}| \left(\|h\|_2 +\|h\|_p+\|u\|_{p}^{p-2}\|h\|_p\right).
\end{align*}
In addition, we see that
\begin{align*}
\left\|\left(A_s -A_{\tilde{s}}\right) \left(2\lambda +(p-1)|u|^{p-2} \right) h\right\|_{L^2(\R^n; |x|^2 \,dx)} 
 \lesssim  |s-\tilde{s}| \left(\|h\|_2 +\|h\|_p+\|u\|_{p}^{p-2}\|h\|_p\right).
\end{align*}
Notice that 
$$
\left\|A_s f\right\|_{L^2(\R^n; |x|^2 \,dx)} \lesssim \|f\|_2, \quad \left\|A_s f\right\|_{L^2(\R^n; |x|^2 \,dx)} \lesssim \|f\|_{\frac{p}{p-1}}.
$$
Further, we can conclude that
\begin{align*}
\left\|A_{\tilde{s}} \left(2 \lambda+\left(|u|^{p-2} -|\tilde{u}|^{p-2} \right)\right) h \right\|_{X_p} &\lesssim \|h\|_2 +\|h\|_p+ \left\||u|^{p-2} -|\tilde{u}|^{p-2}\right\|_{\frac{p}{p-2}}\|h\|_p.
\end{align*}
Note that
$$
\left\||u|^{p-2}-|\tilde{u}|^{p-2}\right\|_{\frac{p}{p-2}} \leq \left\||u-\tilde{u}|^{p-2}\right\|_{\frac{p}{p-2}}=\left\|u-\tilde{u}\right\|_{p}^{p-2}, \quad 2<p \leq 3
$$
and
$$
\left\||u|^{p-2}-|\tilde{u}|^{p-2}\right\|_{\frac{p}{p-2}} \lesssim \left\|\left(|u|^{p-3}+|\tilde{u}|^{p-3}\right)|u-\tilde{u}|\right\|_{\frac{p}{p-2}} \leq \left(\|u\|_p^{p-3} +\|\tilde{u}\|_p^{p-3}\right) \|u -\tilde{u}\|_p, \quad p>3.
$$
Consequently, from the calculations above, \eqref{conti} holds true. This implies that $\frac{\partial F}{\partial u}$ is continuous. By a similar argument, we are also able to show that $\frac{\partial F}{\partial s}$ exists and
$$
\frac{\partial F}{\partial s}=-\left(\left(-\Delta\right)^s \log (-\Delta) \right)\left(\left(-\Delta\right)^s + \left(\omega+|x|^2\right) + 2 \lambda\right)^{-2} \left(2\lambda u +|u|^{p-2}u \right).
$$
In addition, we can prove that  $\frac{\partial F}{\partial s}$. Thus we have that $F$ is of class $C^1$.

Now we employ the implicit function theorem to establish theorem. Note first that $F(u_0, s_0)=0$ and
$$
\frac{\partial F}{\partial u}(u_0, s_0)=1+K, \quad K:=-\left(\left(-\Delta\right)^{s_0} + \left(\omega+|x|^2\right) +2 \lambda \right)^{-1} \left(2\lambda +(p-1)|u_0|^{p-2} \right).
$$
It is simple to see that $K$ is compact on $L^2_{rad}(\R^n)$. Moreover, from Lemma \ref{nondegeneracy}, we have that $-1 \not \in \sigma(K)$. Then $1+ K$ is invertible. Furthermore, arguing as before, we can show that $1+K$ is bounded from $X_p$ to $X_p$. This implies that $(1+K)^{-1}$ is bounded from $X_p$ to $X_p$. It then follows from the implicit function theorem that theorem holds true. This completes the proof.
\end{proof}

In the following, we shall consider the maximum extension of the branch $u_s$ for $s \in [s_0, s_*)$, where $s_*>s_0$ is given by
$$
s_*:=\sup \left\{s_0<\tilde{s}<1, u_s \in C^1([s_0, \tilde{s}); X_p), u_s \,\,\mbox{satisfies the assumptions of Lemma \ref{bifurcation} for} \,\, s \in [s_0, \tilde{s}) \right\}.
$$

\begin{lem} \label{bdd}
There holds that
$$
\int_{\R^n} \left(w + |x|^2\right) |u_s|^2\, dx \sim \int_{\R^n} |(-\Delta)^{\frac s 2} u_s|^2 \, dx \sim \int_{\R^n} |u_s|^p \,dx \sim 1
$$
for any $s \in [s_0, s_*)$.
\end{lem}
\begin{proof}
Define
$$
M_s:=w \int_{\R^n} |u_s|^2\, dx, \quad H_s:= \int_{\R^n} |x|^2 |u_s|^2\, dx, \quad T_s:=\int_{\R^n} |(-\Delta)^{\frac s 2} u_s|^2 \, dx, \quad V_s:=\int_{\R^n} |u_s|^p \,dx.
$$
Since $u_s \in H^s(\R^n)$ is a solution to \eqref{equ}, then
\begin{align} \label{ph1}
T_s + M_s +H_s=V_s.
\end{align} 
In addition, we have that $u_s$ satisfies the following Pohozaev identity,
\begin{align} \label{ph2}
\frac{N-2s}{2} T_s +\frac N 2 M_s + \frac{N+2}{2}H_s=\frac N p V_s.
\end{align}
Combining \eqref{ph1} and \eqref{ph2}, we see that
\begin{align} \label{ph3}
s T_s -H_s=\frac{N(p-2)}{2p} V_s.
\end{align}
It follows from \eqref{ph1} and \eqref{ph3} that
\begin{align*}
s_0 M_s +(1+s_0)H_s \leq  s M_s +(1+s)H_s=\frac{2ps-N(p-2)}{2p} V_s <\frac{2ps_*-N(p-2)}{2p} V_s
\end{align*}
and
\begin{align*}
s_* M_s +(1+s_*)H_s >  s M_s +(1+s)H_s=\frac{2ps-N(p-2)}{2p} V_s \geq \frac{2ps_0-N(p-2)}{2p} V_s.
\end{align*}
Consequently, we have that $M_s + H_s \sim V_s$ for any $s \in [s_0, s_*)$. It follows from \eqref{ph1} and \eqref{ph3} that
\begin{align*}
(1+s_0)T_s \leq (1+s)T_s +M_s=\frac{N(p-2)+2p}{2p}V_s
\end{align*}
and
\begin{align*}
(1+s_*)T_s > (1+s)T_s +M_s=\frac{N(p-2)+2p}{2p}V_s.
\end{align*}
This leads to $T_s \sim V_s$ for any $s \in [s_0, s_*)$. Therefore, we obtain that 
\begin{align} \label{equi}
M_s +H_s \sim T_s \sim V_s
\end{align} 
for any $s \in [s_0, s_*)$. Since $2<p<p_{s_0}$, there exists $0<\theta<1$ such that $p=2\theta + (1-\theta) p_{s_0}$.  From Gagliardo-Nirenberg's inequailty and H\"older's inequality, we then get that
\begin{align} \label{equiv1}
V_s \leq M_s^{\theta} \left(\int_{\R^n}|u_s|^{p_{s_0}}\, dx \right)^{(1-\theta)} &\lesssim \left(M_s +H_s\right)^{\theta} \left(\int_{\R^n}|(-\Delta)^{\frac {s_0} {2}}u_s|^2 \, dx \right)^{\frac{p_{s_0}(1-\theta)}{2}}.
\end{align}
In addition, there holds that
\begin{align} \label{equiv2}
\int_{\R^n}|(-\Delta)^{\frac {s_0} {2}}u_s|^2 \, dx \leq \left(\int_{\R^n}|u_s|^2 \, dx\right)^{\frac{s-s_0}{s}}\left(\int_{\R^n}|(-\Delta)^{\frac {s} {2}}u_s|^2 \, dx \right)^{\frac{s_0}{s}}.
\end{align}
Utilizing \eqref{equi}, \eqref{equiv1} and \eqref{equiv2} then implies that
\begin{align*} 
M_s +H_s \sim T_s \sim V_s \gtrsim 1
\end{align*} 
for any $s \in [s_0, s_*)$. Arguing as the proof of \cite[Lemma 8.2]{FLS}, we can obtain that $V_s \lesssim 1$ for any $s \in [s_0, s_*)$. This in turn implies that
$$
M_s +H_s \sim T_s \sim V_s \lesssim 1
$$
for any $s \in [s_0, s_*)$. This completes the proof.
\end{proof}

\begin{lem} \label{coercivity}
Let $n \geq 1$, $s_0 \leq s \leq 1$, $\omega>-\lambda_{1, s_0}$ and $2<p<2_{s_0}^*$. Suppose that $u_s \in X_p$ is a ground state to \eqref{equ}. Then there exists $\mu_s>0$ such that
\begin{align} \label{coercive}
\liminf_{\sigma \to s^-} \mathcal{L}_{+, \sigma} \mid_{\{u_{\sigma}\}^{\bot}} \geq \mu_s.
\end{align}
\end{lem}
\begin{proof}
Define
\begin{align} \label{cmin}
\alpha_s:=\inf \left\{\langle \mathcal{L}_{+,s} f, f \rangle : f \bot u_s, \|f\|_2=1\right\}.
\end{align}
Obviously, we have that $\alpha_s \geq 0$. First we shall verify that $\alpha_s>0$ is attained. Let $\{f_k\}$ be a minimizing sequence to \eqref{cmin} such that $f_k \bot u_s$, $\|f_k\|_2=1$ and $\langle \mathcal{L}_{+,s} f_k, f_k \rangle=\alpha_s+o_k(1)$. 
Observe that $\{f_k\}$ is bounded in $\Sigma_s$. Therefore, there exists a function $f \in \Sigma_s$ such that $f_k \wto f$ in $\Sigma_s$ and $f_k \to f$ in $L^q(\R^n)$ for any $q \in [2, 2_s^*)$ as $n \to \infty$. This leads to $f \bot u_s$, $\|f\|_2=1$ and $\langle \mathcal{L}_{+,s} f, f \rangle=\alpha_s$. Contrarily, we assume that $\alpha_s=0$. When $s<1$, using the fact that $Ker [\mathcal{L}_{+, s}]=\{0\}$ by Lemma \ref{nondegeneracy} and arguing as the proof of \cite[Proposition 6]{SS}, we are able to reach a contradiction. This in turn shows that $\alpha_s>0$ and 
$$
\langle \mathcal{L}_{+,s} u, u \rangle  \geq \alpha_s  \|u\|_2^2, \quad \forall \,\, u \bot u_s.
$$
While $s =1$, using the fact that $Ker[\mathcal{L}_{+, 1}]=\{0\}$ and following the spirit of the proof of \cite[Proposition 6]{SS}, we can also derive that $\alpha_1>0$ and
$$
\langle \mathcal{L}_{+,1} u, u \rangle  \geq \alpha_1  \|u\|_2^2, \quad \forall \,\, u \bot u_1.
$$
Thus the proof is completed.
\end{proof}

\begin{lem} \label{positive}
Let $u_{s_0}>0$ be a solution to \eqref{equ} with $s=s_0$. Then, for any $s \in [s_0, s_*)$, there holds that
$u_s(x)>0$ for $x \in \R^n$ and $u_s(x) \lesssim |x|^{-n}$ for $|x| \gtrsim 1$.
\end{lem}
\begin{proof}
In the spirit of the proof of \cite[Lemma 8.3]{FLS}, we need to verify that the operator $\mathcal{L}_{-, s}$ enjoys the Perron-Frobenius type property, where
$$
\mathcal{L}_{-,s}:=(-\Delta)^s + \left(\omega+|x|^2\right) -|u|^{p-2}.
$$ 
In addition, we need to check that $\mathcal{L}_{-, \tilde{s}} \to \mathcal{L}_{-,s}$ as $\tilde{s} \to s$ in norm resolvent sense. 

Define $H:=(-\Delta)^s+|x|^2$, which generates a semigroup $e^{-t H}$ with positive integral kernel. Then we have that $e^{-t H}$ acting on $L^2(\R^n)$ is positivity improving. Next we show that $w+|u|^{p-2}$ belongs to Kato class, i.e.
\begin{align} \label{kato}
\lim_{\lambda \to \infty} \left\|(H+\lambda)^{-1} \left(\omega+|u|^{p-2}\right)\right\|_{L^{\infty} \to L^{\infty}}=0.
\end{align}
Note that $H+ \lambda>(-\Delta)^s + \lambda$, then
$$
\left(H+ \lambda\right)^{-1}< \left((-\Delta)^s +\lambda\right)^{-1}.
$$
Let $\mathcal{K}$ be the fundamental solution to the equation
$$
(-\Delta)^s u+\lambda u=0.
$$
Then we have that
$$
\mathcal{K}(x)=\int_0^{+\infty} e^{-\lambda t} \mathcal{H}(x, t) \, dt,
$$
where
$$
\mathcal{H}(x, t):=\int_{\R^n} e^{2 \pi \textnormal{i} x \cdot \xi -t |\xi|^{2s}} \, d \xi.
$$
From $(A 4)$ in \cite[Appendix A]{FQT}, we find that
$$
0 <\mathcal{H}(x ,t) \lesssim \min \left\{t^{-\frac{n}{2s}}, t|x|^{-n-2s}\right\}.
$$
This gives that, for any $q \geq 1$,
\begin{align*}
\left\|\mathcal{H}\right\|_q \lesssim \left(\int_{|x| \leq t^{\frac {1}{2s}}} t^{-\frac{nq}{2s}}\, dx \right)^{\frac 1 q} + \left(\int_{|x| \geq t^{\frac {1}{2s}}} t^q |x|^{-(n+2s)q} \, dx \right)^{\frac 1 q}  \lesssim t^{-\frac{n}{2s}\left(1-\frac 1 q\right)}.
\end{align*}
It then follows that
\begin{align*}
\left\|\mathcal{K}\right\|_q \leq \int_0^{+\infty} e^{-\lambda t} \left\|\mathcal{K}(\cdot, t)\right\|_q \, dt \lesssim \int_0^{+\infty} e^{-\lambda t} t^{-\frac{n}{2s}\left(1-\frac 1 q\right)} \, dt \lesssim \lambda^{\frac{n}{2s}\left(1-\frac 1 q\right)-1},
\end{align*}
where $q \geq 1$ satisfies 
$$
\frac{n}{2s}\left(1-\frac 1 q\right)<1.
$$
Using Young's inequality, we then get that, for any $f \in L^{\infty}(\R^n)$,
$$
\|\left((-\Delta)^s +\lambda\right)^{-1}\left(\omega+|u|^{p-2}\right)f\|_{\infty} \lesssim \lambda^{-1}\omega\|f\|_{\infty} + \lambda^{\frac{n}{2s}\left(1- \frac 2 q\right)-1} \|f\|_{\infty},
$$
which readily yields that 
$$
\|\left((-\Delta)^s +\lambda\right)^{-1}\left(\omega+|u|^{p-2}\right)\|_{L^{\infty}\to L^{\infty}}=o_{\lambda}(1). 
$$
Thus \eqref{kato} holds true and the desired result follows. Arguing as the proof of \cite[Lemma C.2]{FL}, we conclude that the operator $\mathcal{L}_{-, s}$ enjoys Perron-Frobenius type property.

Next we prove the convergence of the operator in norm resolvent sense. Observe first that
\begin{align*}
\mathcal{L}_{-,\tilde{s}} + z&=(-\Delta)^s + \left(\omega+|x|^2\right) -|u|^{p-2} +z + (-\Delta)^{\tilde{s}}-(-\Delta)^s \\
&= \left(1+ \left((-\Delta)^{\tilde{s}}-(-\Delta)^s\right)\left(\mathcal{L}_{-,s} + z\right)^{-1}\right)\left(\mathcal{L}_{-,s} + z\right).
\end{align*}
Therefore, we have that
\begin{align*}
\left(\mathcal{L}_{-,s} + z\right)^{-1}-\left(\mathcal{L}_{-, \tilde{s}} + z\right)^{-1}=\left(\mathcal{L}_{-,s} + z\right)^{-1} \left(1-\left(1+ \left((-\Delta)^{\tilde{s}}-(-\Delta)^s\right)\left(\mathcal{L}_{-,s} + z\right)^{-1}\right)^{-1}\right).
\end{align*}
As the proof of Lemma \ref{monotonicity}, we can show that
$$
\left\|\left(\mathcal{L}_{-,s} + z\right)^{-1}-\left(\mathcal{L}_{-, \tilde{s}} + z\right)^{-1}\right\|_{L^2 \to L^2} \to 0, \quad \mbox{as} \,\, \tilde{s} \to s.
$$
This indicates that $\mathcal{L}_{-, \tilde{s}} \to \mathcal{L}_{-, s}$ in the norm-resolvent sense as $\tilde{s} \to s$. Thus the proof is completed.
\end{proof}

\begin{lem} \label{conv}
Let $\{s_n\} \subset [s_0, s_*)$ be a sequence such that $s_n \to s_*$ as $n \to \infty$ and $u_{s_n}>0$ for any $n \in \mathbb{N}$. Then there exists $u_* \in X_p$ such that $u_{s_n} \to u_*$ in $X_p$ as $n \to \infty$. Moreover, there holds that $u_*>0$ and it solves the equation
\begin{align} \label{equ2}
(-\Delta)^{s_*} u_*+ \left(\omega+|x|^2\right)  u_*=u_*^{p-1}.
\end{align}
\end{lem}
\begin{proof}
From Lemma \ref{bdd}, we know that $u_{s_n}$ is bounded in $\Sigma_{s_0}$. Thus there exists $u_* \in \Sigma_{s_0}$ such that $u_{s_n} \wto u_*$ in $\Sigma_{s_0}$ and $u_{s_n} \to u_*$ in $L^q(\R^n)$ for any $q \in [2, 2_{s_0}^*)$. Since $u_{s_n} >0$, then $u_* \geq 0$. It follows from Lemma \ref{bdd} that $u_* \neq 0$. Note that
$$
u_{s_n}=\left((-\Delta)^{s_n}+ \left(\omega+|x|^2\right) + 2 \lambda\right)^{-1} \left(2 \lambda u_{s_n} + u_{s_n}^{p-1}\right).
$$
Since $u_{s_n} \to u_*$ in $L^2(\R^n) \cap L^p(\R^n)$ as $n \to \infty$, then
$$
u_*=\left((-\Delta)^{s_*}+ \left(\omega+|x|^2\right) + 2 \lambda\right)^{-1} \left(2 \lambda u_{*} + u_{*}^{p-1}\right).
$$
This implies that $u_*$ solves \eqref{equ2} and $u_{s_n} \to u_*$ in $X_p$ as $n \to \infty$. Thus the proof is completed.
\end{proof}

\begin{lem} \label{extend}
Let $u_0 \in X_p$ be a ground state to \eqref{equ} with $s=s_0$. Then its maximum branch $u_s$ with $s \in [s_0, s_*)$ extends to $s_*=1$.
\end{lem}
\begin{proof}
Define 
$$
\mathcal{L}_{+, s}:=(-\Delta)^s+ \left(\omega+|x|^2\right) -(p-1) |u_s|^{p-2}.
$$
Reasoning as the proof of the norm-resolvent convergence of $\mathcal{L}_{-, s}$ in Lemma \ref{positive}, we can also show that $\mathcal{L}_{+, \tilde{s}} \to \mathcal{L}_{+,s}$ in the norm-resolvent sense as $\tilde{s} \to s$. This gives that
$$
\mathcal{N}_{-, rad}(\mathcal{L}_{+, s})=\mathcal{N}_{-, rad}(\mathcal{L}_{+, s_0})=1, \quad s \in [s_0, s_*).
$$
Let $\{s_n\} \subset [s_0, s_*)$ be such that $s_n \to s_*$. Since $u_0 \in X_p$ is a ground state to \eqref{equ} with $s=s_0$, then $u_0>0$. In view of Lemma \ref{positive}, then $u_{s_n}>0$. From Lemma \ref{conv}, we know that there exists $u_*>0$ solving \eqref{equ2}. Note that $\mathcal{L}_{+, s_n} \to \mathcal{L}_{+, s_*}$ in the norm-resolvent sense as $n \to \infty$. By the lower semicontinuity of the Morse index, we have that
$$
1=\liminf_{n \to\infty} \mathcal{N}_{-, rad}(\mathcal{L}_{-, s_n}) \geq \mathcal{N}_{-, rad}(\mathcal{L}_{+, s_*}).
$$
This implies that $\mathcal{N}_{-, rad}(\mathcal{L}_{+, s_*}) \leq 1$. On the other hand, since $u_*$ solves \eqref{equ2}, then we see that
$$
\langle u_*, \mathcal{L}_{+, s_*} u_* \rangle=-(p-2) \int_{\R^n} |u_*|^p \, dx <0.
$$
Thus we conclude that $\mathcal{N}_{-, rad}(\mathcal{L}_{+, s_*})=1$, which yields that $u_*$ is a ground state to \eqref{equ2}. As a result, we have that $s_*=1$. On the other hand, by the nondegeneracy of $\mathcal{L}_{+, s_*}$, then $u_s$ can be extended beyond $s_*$. This is impossible and the proof is completed.
\end{proof}

Now we are ready to prove Theorem \ref{thm1}.

\begin{proof}[Proof of Theorem \ref{thm1}]
Let $n \geq 1$, $0<s_0<1$ and $2<p<2_{s_0}^*$. Let $u_{s_0}>0$ and $\tilde{u}_{s_0}>0$ be two different ground states to \eqref{equ} with $s=s_0$, which are indeed radially symmetric. From Lemma \ref{nondegeneracy}, we obtain that the associated linearized operators around $u_{s_0}$ and $\tilde{u}_{s_0}$ are nondegenerate. Then, by Lemmas \ref{bifurcation} and \ref{extend}, we have that $u_s \in C^1([s_0, 1); X_p)$ and $\tilde{u}_s \in C^1([s_0, 1); X_p)$. Moreover, by the local uniqueness of solutions derived in Lemma \ref{bifurcation}, we get that $u_s \neq \tilde{u}_s$ for any $s \in [s_0, 1)$. It follows from Lemma \ref{conv} that there exist $u_* \in X_p$ and $\tilde{u}_* \in X_p$ such that $u_s \to u_*$ and $\tilde{u}_{s} \to \tilde{u}_*$ in $X_p$ as $s \to 1^-$. In addition, $u_*>0$ and $\tilde{u}_*>0$ solve \eqref{equ2} with $s_*=1$. Thanks to \cite[Theorem 1.3]{HO1} and \cite[Theorem1.2]{HO2}, then we have that $u_*=\tilde{u}_*$. This implies that $\|u_s-\tilde{u}_s\|_{X_p} \to 0$ as $s \to 1^-$.
Note that the linearized operator $\mathcal{L}_{+, 1}$ around $u_*$ is nondegenerate, see \cite[Theorem 0.2]{KT}. Remark that, from the proof of \cite[Theorem 0.2]{KT}, it is simple to see that the result also holds true for $n=1$. Then, by the implicit function theorem, there exists a unique branch $\hat{u}_{s} \in C^1((1-\delta, 1]; X_p)$ solving \eqref{equ} with $\hat{u}_1=u^*$ for some $\delta>0$. This contradicts with $u_s \neq \tilde{u}_s$ for any $s\in [s_0, 1)$. Thus the proof is completed.
\end{proof}


\begin{thebibliography}{99}

\bibitem{DH} Z. Ding, H. Hajaiej, \emph{On a fractional Schr\"odinger equation in the presence of harmonic potential}, Electron. Res. Arch. 29 (2021), no. 5, 3449--3469. 

\bibitem{FQT} P. Felmer, A.  Quaas, J. Tan, \emph{Positive solutions of nonlinear Schr\"odinger equation with the fractional Laplacian},  Proc. Roy. Soc. Edinburgh Sect. A 142 (2012), no. 6, 1237--1262.

\bibitem{FL} R.L. Frank,  E. Lenzmann, \emph{Uniqueness of non-linear ground states for fractional Laplacians in $\R$},  Acta Math. 210 (2013), no. 2, 261--318. 

\bibitem{FLS} R.L. Frank,  E. Lenzmann, L. Silvestre, \emph{Uniqueness of radial solutions for the fractional Laplacian}, Comm. Pure Appl. Math. 69 (2016), no. 9, 167--1726.

\bibitem{HO1} M. Hirose, M. Ohta, \emph{Structure of positive radial solutions to scalar field equations with harmonic potential}, J. Differential Equations 178 (2002), no. 2, 519--540.

\bibitem{HO2} M. Hirose, M. Ohta, \emph{Uniqueness of positive solutions to scalar field equations with harmonic potential}, Funkcial. Ekvac. 50 (2007), no. 1, 67--100.

\bibitem{KT} Y. Kabeya, K. Tanaka, \emph{Uniqueness of positive radial solutions of semilinear elliptic equations in $\R^n$ and S\'er\'e's non-degeneracy condition}, Comm. Partial Differential Equations 24 (1999), no. 3-4, 563--598.

\bibitem{L1} N. Laskin, \emph{Fractional quantum mechanics and L\'evy path integrals}, Phys. Lett. A 268 (2000), no. 4-6, 298--305. 

\bibitem{L2} N. Laskin, \emph{Fractional Schrödinger equation}, Phys. Rev. E (3) 66 (2002), no. 5, 056108, 7 pp.

\bibitem{SS} M. Stanislavova, A.G. Stefanov,  \emph{Ground states for the nonlinear Schrödinger equation under a general trapping potential}, J. Evol. Equ. 21 (2021), no. 1, 671--697. 

\bibitem{W} Michel Willem, \emph{Minimax Theorems}, Birkhauser Verlag, Boston, Basel, Berlin, 1996.

\end{thebibliography}
\end{document}